\def\N{\mathbb N}
\def\Z{\mathbb Z}
\def\Q{\mathbb Q}
\def\R{\mathbb R}

\def\A{\mathcal A}
\def\B{\mathcal B}
\def\S{\mathcal S}
\def\I{\mathcal I}
\def\pfz{\begin{proof}}
\def\pfk{\end{proof}}
\def\bul{{\scriptstyle\bullet}}

\documentclass{elsarticle}

\usepackage[cp1250]{inputenc}
\usepackage[IL2]{fontenc}
\usepackage[czech,english]{babel}

\usepackage{amssymb,amsmath,amsfonts,amsthm,multirow}
\usepackage{cancel,enumerate,stmaryrd}
\usepackage{multirow,bigstrut,bbding}
\usepackage[nottoc]{tocbibind}

\newtheorem{lem}{Lemma}
\newtheorem{thm}[lem]{Theorem}
\newtheorem{prop}[lem]{Proposition}
\newtheorem{coro}[lem]{Corollary}
\newtheorem{de}[lem]{Definition}

\newtheorem{pozn}[lem]{Remark}

\begin{document}

\begin{frontmatter}

\title{Description of spectra of quadratic Pisot units}

\author{Zuzana Mas\'akov\'a, Kate\v rina Pastir\v c\'akov\'a, Edita Pelantov\'a}
\address{Department of Mathematics FNSPE, Czech Technical University in Prague\\ 
Trojanova 13, 120 00 Praha 2, Czech Republic}

\cortext[]{Corresponding author: Zuzana Mas\'akov\'a (\tt{zuzana.masakova@fjfi.cvut.cz})}

\date{\today}

\begin{abstract}
The spectrum of a real number $\beta>1$ is the set $X^{m}(\beta)$
of $p(\beta)$ where $p$ ranges over all polynomials with
coefficients restricted to $\A=\{0,1,\dots,m\}$. For a quadratic
Pisot unit $\beta$, we determine the values of all distances
between consecutive points and their corresponding frequencies, by
recasting the spectra in the frame of the cut-and-project scheme.
We also show that shifting the set $\A$ of digits so that it
contains at least one negative element, or considering negative
base $-\beta$ instead of $\beta$, the gap sequence of the
generalized spectrum is a coding of an exchange of three
intervals.
\end{abstract}


\end{frontmatter}


\section{Introduction}

The spectrum of a real number $\beta>1$ is the set of $p(\beta)$
where $p$ ranges over all polynomials with coefficients restricted
to a finite set of consecutive integers, in particular,
\begin{equation}\label{eq:spectrum}
X^{m}(\beta) = \Big\{\sum_{j=0}^n a_j \beta^j : n\in\N,\, a_j\in\{0,1,\dots,m\}\Big\}\,.
\end{equation}
The study of such sets for $\beta\in(1,2)$ and $m=1$ was initiated
by Erd\H{o}s et al. in 1990~\cite{ErdosJooKomornik1990}. Their
interest~\cite{ErdosJooKomornik1998,ErdosKomornik1998} is to study
the difference sequence $(y_{k+1}-y_k)_{k\in\N}$, where
$X^m(\beta)=\{0=y_0<y_1<y_2<\cdots\}$, in particular the values
$$
l^m(\beta) = \liminf_{k\to\infty} (y_{k+1} - y_k)\quad\text{ and }\quad L^m(\beta) = \limsup_{k\to\infty} (y_{k+1} - y_k)\,.
$$
The relevance of Pisot numbers in the problem of spectra was
indicated by Bugeaud in 1996~\cite{Bugeaud1996} who showed that a real
$\beta\in(1,2)$ is a Pisot number if and only if $l^m(\beta)>0$
for all $m\geq 1$. An important recent contribution is due to
Feng~\cite{Feng} who shows that $l^m(\beta)$ is positive if and
only if $\beta$ is a Pisot number or $m< \beta-1$. Many
other authors have contributed to the problem; for an exhaustive
overview, see the paper by Akiyama and
Komornik~\cite{AkiyamaKomornik2013}.

A particular question is to determine the exact values of
$l^m(\beta)$, $L^m(\beta)$. First to give such a result for
general $m\geq 1$ were Komornik et al.\ in
2000~\cite{KomornikLoretiPedicini2000} who provided $l^m(\beta)$
for the golden ratio $\beta=\frac12(1+\sqrt5)$. Komatsu in
2002~\cite{Komatsu} and independently Borwein and Hare in
2003~\cite{BorweinHare2003} extended the result to all quadratic
Pisot units. Their method, however, does not testify about the
structure of the gap sequence $(y_{k+1}-y_k)_{k\in\N}$, in
particular, nor about $L^m(\beta)$.

In 2002, Bugeaud~\cite{Bugeaud2002} provides a substitution that
can be used for generating the difference sequence for the
spectrum $X^1(\beta)$ of the so-called $d$-bonacci numbers $\beta>1$,
zeros of $x^d-x^{d-1}-\cdots-x-1$. This allows him
to determine all gaps and the corresponding frequencies. In the
same year, Feng and Wen~\cite{FengWen02} showed that for any Pisot
number $\beta$ and $m>\beta-1$, the sequence of distances
$(y_{k+1}-y_{k})_{k\in\N}$ in $X^{m}(\beta)$ can be generated by a
substitution over a finite alphabet. Their proof is constructive,
however, does not provide any explicit prescription for the
substitution nor for the values of distances and their
frequencies. Moreover, the cardinality of the alphabet of the
substitution found by their construction grows rapidly with $m$.
The construction of Feng and Wen was used in 2006 by Garth and
Hare~\cite{GarthHare2006} for determining the substitution, gap
sizes and their frequencies explicitly for the spectra
$X^{\lfloor\beta\rfloor}(\beta)$, where $\beta>1$ is a zero
of $x^d-px^{d-1}-\cdots -px - q$, $p\geq q\geq 1$.

Our main result concerns the distance sequence for the spectra
$X^m(\beta)$ of quadratic Pisot units $\beta$. Unlike the previous
results, we consider arbitrary $m\in\N$, $m>\beta-1$. We show that
the distances $y_{k+1}-y_k$ take (up to finitely many exceptions)
only three values.

%

\begin{thm}\label{t:1}
Let $\beta$ be a quadratic Pisot unit, $m\in\N$, $m>\beta-1$. Then
there exist $\Delta_1$, $\Delta_2>0$ such that the distances
$y_{k+1}-y_{k}$ between consecutive points in $X^m(\beta)$ take
values in $\{\Delta_1,\Delta_2,\Delta_1+\Delta_2\}$, up to
finitely many exceptions.
\end{thm}

The explicit values of $\Delta_1, \Delta_2$ and the frequencies of
the gaps $\Delta_1,\Delta_2,\Delta_1+\Delta_2$, dependently on
$m$, are given in Theorem~\ref{t:explicit}. Let us mention that
for the particular case of the golden ratio, and any $m\geq 1$, this
was already given in~\cite{Hare2004}. However, the proof given by
the author contains a flaw, as explained in
Section~\ref{sec:flaw} of the present paper.

Our method in proof of Theorem~\ref{t:1} is the use of infinite
words coding exchange of three intervals, the so-called 3iet
words. More precisely, we show that the sequence of gaps in
$X^m(\beta)$ `almost' coincides with some 3iet word. Although the
coincidence is broken at infinitely many places (see
Proposition~\ref{p:nekonecneporuseni}), the frequency of such
perturbations is 0. This is the reason why the values for gaps and
their frequencies given in~\cite{Hare2004} are correct.

It is interesting to mention that a perfect coincidence with 3iet
words can be achieved when considering a slightly generalized
problem, in particular, when shifting the alphabet to contain both
positive and negative digits, or considering negative base. With
such assumptions, the spectrum is distributed over all the real
line. Let $\alpha$ be real, $|\alpha|>1$, and let $\A\ni0$ be a
finite set of consecutive integers. Set
$$
X^{\A}(\alpha) = \Big\{\sum_{j=0}^n a_j \alpha^j : n\in\N,\, a_j\in\A\Big\}\,.
$$
With such a notation, we will prove the following statement.

\begin{thm}\label{t:2}
Let $\beta$ be a quadratic Pisot unit, and let $\A\ni0$ be a finite set of consecutive integers, $\#\A>\beta$.  Let
$$
\alpha = -\beta\quad\text{ or }\quad
\alpha = \beta \text{ and } \{-1, 0,1\} \in\A\,.
$$
Then there exist $\Delta_1$, $\Delta_2>0$ such that the distances
between consecutive points in $X^\A(\alpha)$ take values in
$\{\Delta_1,\Delta_2,\Delta_1+\Delta_2\}$. Moreover, the
bidirectional sequence of gaps in $X^\A(\alpha)$ is a coding of
exchange of three intervals.
\end{thm}

The correspondence of spectra with 3iet words is established using
the so-called cut-and-project scheme, introduced in the following
section.

\section{Cut-and-project sets}\label{sec:cap}

When $\beta$ is an algebraic integer of degree $d$, the elements
of the spectra are clearly expressible as elements of the set
$\Z[\beta]=\Z+\Z\beta+\cdots+\Z\beta^{d-1}$. If, moreover, $\beta$
is a Pisot number (algebraic integer $>1$ with conjugates of modulus strictly less than $1$), one can easily show that the Galois image of
the spectrum of a Pisot number is bounded.

In the following proposition, we determine the interval in which
the Galois image of the spectrum is contained in case that $\beta$
is a quadratic Pisot unit.

\begin{prop}\label{p:XY}
Let $\beta$ be a quadratic Pisot unit, $m\in\N$.  Then
$$
X^{m}(\beta)\subset \{x\in\Z[\beta] : x'\in\Omega\}\,,
$$
where $x'$ stands for the Galois image of $x$ in the field $\Q(\beta)$ and
$$
\begin{aligned}
\Omega &=
\Big(-\frac{m\beta}{\beta^2-1},\frac{m\beta^2}{\beta^2-1}\Big)  \text{ when } \beta^2=p\beta+1\,,\text{ and }\\[2mm]
\Omega &=
\Big[0,\frac{m\beta}{\beta-1}\Big)  \text{ when } \beta^2=p\beta-1\,.
\end{aligned}
$$
\end{prop}

\pfz
Take $x=\sum_{i=0}^na_i\beta^i\in X^m(\beta)$. Obviously, $x\in\Z[\beta]=\Z+\Z\beta$. Taking the Galois image, we have
$x'=\sum_{i=0}^na_i{\beta'}^i$. If $\beta$ satisfies $\beta^2=p\beta+1$, we have $\beta'=-\frac1\beta$, and thus
$$
-\frac{m\beta}{\beta^{2}-1}=\sum_{i=0}^\infty \frac{m}{-\beta^{2i+1}}<x'=\sum_{i=0}^n\frac{a_i}{(-\beta)^i} < \sum_{i=0}^\infty \frac{m}{\beta^{2i}}=
\frac{m\beta^2}{\beta^2-1}\,.
$$
The case $\beta^2=p\beta-1$ is proven similarly, taking into account that $\beta'=\frac1\beta$.
\pfk

The spectrum $X^{m}(\beta)$ lies in $\Z[\beta]$, and the pairs
$(x,x')$ for $x\in\Z[\beta]$ form a lattice. The above proposition
states that the pairs $(x,x')$ with $x\in X^{m}(\beta)$ belong to
a strip of a bounded width cut from the lattice. Such
considerations give a motivation for the definition of a cut-and-project set.

\begin{de}\label{d:cap}
Let $\varepsilon,\eta$ be irrational, $\varepsilon\neq\eta$ and
$\Omega$ a bounded interval. Let
$\star:\Z[\eta]\to\Z[\varepsilon]$ be the isomorphism between
additive groups $\Z[\eta]=\Z+\Z\eta$ and
$\Z[\varepsilon]=\Z+\Z\varepsilon$ given by
$(a+b\eta)^\star=a+b\varepsilon$. The set
$$
\Sigma_{\varepsilon,\eta}(\Omega) = \{ x\in\Z[\eta] : x^\star
\in\Omega \}
$$
is called a cut-and-project set with acceptance interval $\Omega$.
\end{de}

Note that the above definition is a very special case of a rather
general concept~\cite{Moody} of model sets arising by projection
of a $d$-dimensional lattice to a suitably oriented
lower-dimensional subspace of $\R^d$.
As shown in~\cite{GuMaPeBordeaux}, the sequence of gaps in a
cut-and-project set can be generated using the transformation of
exchange of 2 or 3 intervals.

\begin{de}\label{d:3iet}
Let \ $0<\lambda\leq \mu<1$. The transformation $T:[0,1)\to[0,1)$
defined by
$$
T(x)=\begin{cases}
x+1-\lambda  & \text{ for }x\in[0,\lambda)=:I_A\\
x+1-\lambda-\mu & \text{ for }x\in[\lambda,\mu)=:I_B\\
x-\mu & \text{ for }x\in[\mu,1)=:I_C
\end{cases}
$$
is called an exchange of 3 intervals, if $\lambda<\mu$, (and
exchange of 2 intervals, when $\lambda=\mu$).
\end{de}

For an arbitrary $\rho\in[0,1)$, the orbit of $\rho$ under $T$ can
be coded by a bidirectional infinite word ${\bf u}=\cdots
u_{-2}u_{-1}u_0u_1u_2\cdots$ over a ternary alphabet, say
$\{A,B,C\}$, (or a binary alphabet $\{A,C\}$), in a natural way,
$$
u_n=X\quad\text{ if }T^n(\rho)\in I_X\,.
$$
The infinite word ${\bf u}$ is called a 3iet word. If the
parameters $1-\lambda$ and $\mu$ are linearly independent over $\Q$,
then the frequency of letters $A,B,C$ in the infinite word ${\bf
u}$ is equal to the length of the corresponding intervals
$I_A,I_B,I_C$, respectively. The following is a result
of~\cite{GuMaPeBordeaux} describing the distances between
consecutive elements of a cut-and-project set and their ordering.

\begin{prop}\label{p:cap3}
Let $\varepsilon,\eta$ be irrational, $\varepsilon\neq\eta$ and
$\Omega$ a bounded interval. Then there exist positive
$\Delta_1,\Delta_2\in\Z[\eta]$, satisfying
$\Delta_2^\star<0<\Delta_1^\star$, such that the distances between
consecutive elements of $\Sigma_{\varepsilon,\eta}(\Omega)$ take
values in $\{\Delta_1,\Delta_2,\Delta_1+\Delta_2\}$. Moreover, if
the gaps $\Delta_1$, $\Delta_2$, and $\Delta_1+\Delta_2$ are coded
by letters $A,C$, and $B$ respectively, then we obtain a 3iet word
coding an exchange $T$ of three intervals with parameters
$\lambda=1-{\Delta_1^\star}/{|\Omega|}$,
$\mu={-\Delta_2^\star}/{|\Omega|}$, where $|\Omega|$ stands for
the length of the interval $\Omega$.
\end{prop}

The distances between consecutive points of
$\Sigma_{\varepsilon,\eta}(\Omega)$ take always two or three
values which are given in terms of the continued fractions of
$\varepsilon, \eta$, dependently on the width of the interval
$\Omega$, but not on its position on the real line. The distances
are two for a discrete set of values of the width of $\Omega$, and
three otherwise.

The spectra studied in this paper will be put into connection with
cut-and-project sequences where the parameters $\varepsilon, \eta$
are mutually conjugated quadratic units $\eta=\beta$,
$\varepsilon=\beta'$, and the isomorphism $\star$ of the additive
groups $\Z[\varepsilon]$, $\Z[\eta]$ is the Galois automorphism on
the field $\Q(\beta)$. For simplicity, we denote the corresponding
cut-and-project set $\Sigma_{\beta',\beta}(\Omega) =
\Sigma_{\beta}(\Omega)$. In such a case the dependence of
distances on the width of the interval $\Omega$ is expressed in a
more explicit form, and we will use it in
Section~\ref{sec:frequencies} to provide the values of distances
and the corresponding frequencies in the spectra.

\section{Spectrum is not equal to a cut-and-project set}\label{sec:flaw}

Cut-and-project sets were implicitly used in~\cite{Hare2004} for
the description of spectrum $X^m(\tau)$ where
$\tau=\frac12(1+\sqrt5)$. The proof of the result stands on a formula (given in~\cite{Hare2004} just before Definition~4), saying
that a sufficiently large real number $y$ belongs to the
spectrum $X^m(\tau)$ if and only if $y$ belongs to
$\Sigma_\tau(\Omega)$ with $\Omega$ given in
Proposition~\ref{p:XY}. Such a statement is, however, true only if
$m=1$. The aim of this section is to prove the following
proposition.

\begin{prop}\label{p:nekonecneporuseni}
Let $\beta>1$ be a quadratic unit, $m\in\N$, $m\geq
\lfloor\beta\rfloor$, and let $\Omega$ be as in
Proposition~\ref{p:XY}. Then there exist infinitely many $y>0$
such that $y\in\Sigma_\beta(\Omega)$ and $y\notin X^m(\beta)$,
unless $m=\lfloor\beta\rfloor$ and $\beta^2=p\beta+1$ for $p\geq
1$.
\end{prop}

The statement of the proposition follows from Corollaries~\ref{c:1} and~\ref{c:2}, which we provide below.
Let us first present a suitable notation for
elements of the spectrum $X^m(\beta)$, i.e.\ numbers which written
in base $\beta$ use only non-negative powers. For simplifying the
notation of $y=\sum_{k=0}^{n-1}y_k\beta^k\in X^m(\beta)$ we write
symbolically
$$
y=y_{n-1}y_{n-2}\cdots y_1y_0{\scriptstyle\bullet}
$$
and speak about a representation of $y$. In all this section, the
base $\beta$ and the parameter $m$ are fixed, therefore we need
not include them in the notation of the representation. Let us
mention that a more detailed explanations to number
representations will be given (and needed) in the following
section.

A number $y\in X^m(\beta)$ can have many representations. For a
quadratic unit base $\beta>1$, we will use representations of a
special form, as given in Lemmas~\ref{l:Katka1}
and~\ref{l:Katka2}.

The string of digits $y_{n-1}y_{n-2}\cdots y_1y_0$ can be regarded
as a finite word of length $n$ over the alphabet $\A$. The set of
finite words over $\A$ is denoted by $\A^*$. Together with the
operation of concatenation (where the neutral element is the empty
word $\epsilon$), it is a monoid. We denote by $w^i$ the
concatenation of $i$ copies of a finite word $w$, and by
$w^\omega$ an infinite repetition of $w$.
 The monoid $\A^*$ is linearly
ordered by the usual lexicographical order $\prec_{\text{\tiny
lex}}$. It is natural to extend the lexicographic order to
infinite words over $\A$ which form a set denoted by $\A^{\N}$.

We distinguish two cases of bases. First we consider
$\beta^2=p\beta-1$, $p\geq 3$.

\begin{lem}\label{l:Katka1}
 Let $\beta>1$ be a root of $x^2-px+1$, $p\geq 3$, and let
 $m\in\N$, $m\geq \lfloor\beta\rfloor=p-1$. Then every $y\in
 X^m(\beta)$ has in base $\beta$ a
 representation $y=ucm^j{\scriptstyle\bullet}$, where
\begin{enumerate}
 \item $c\in\{0,1,\dots,m-1\}$\,,
 \item $u\in \{0,1,\dots,p-1\}^*$\,,
 \item any suffix of $u$ is lexicographically smaller than the
 infinite eventually periodic word $(p-1)(p-2)^\omega$.
\end{enumerate}
 \end{lem}

\begin{proof}
As $\beta^2+1=p\beta$, we have
\begin{itemize}
 \item[(i)] $z=0p0\bul\ \Rightarrow\ z=101\bul$\,,
 \item[(ii)] $z=0(p-1)(p-2)^k(p-1)0\bul \ \Rightarrow\
       z=10^{k+2}1\bul$ \ for all $k\in\N$.
\end{itemize}
Note that the representation of $z$ on the right in the
implications has strictly smaller sum of digits than the
representation of $z$ on the left. Let us demonstrate that by using
the rules (i) and (ii) one can reduce the sum of digits in a
general representation of a number $y=\sum_{j=0}^ny_j\beta^j\in
X^m(\beta)$, where $y_j\in\{0,1,\dots,m\}$.

If there is an $i\in\N$ such that $y_i,y_{i+2}<m$ and $y_{i+1}\geq
p$, then by (i), the number $y_{i+2}y_{i+1}y_i\bul$ has in the
alphabet $\{0,1,\dots,m\}$ also the representation
$(y_{i+2}+1)(y_{i+1}-p)(y_i+1)\bul$, and thus also $y$ has another
representation over $\{0,1,\dots,m\}$ with strictly smaller digit
sum. Similarly, if there exist $i,k\in\N$ such that
$y_i,y_{i+k+3}<m$, $y_{i+1},y_{i+k+2}\geq p-1$ and $y_{i+j}\geq
p-2$ for $j=2,3,\dots,k+1$, then by (ii),
$$
(y_{i+k+3}+1)(y_{i+k+2}-p+1)(y_{i+k+1}-p+2)\cdots
(y_{i+2}-p+2)(y_{i+1}-p+1)(y_i+1)\bul
$$
is a representation of $y_{i+k+3}\cdots y_{i}\bul$ with strictly
smaller digit sum, and therefore $y$ has another representation
with strictly smaller digit sum.

Repeated application of rules (i) and (ii) yields a
representation of $y\in X^m(\beta)$ in the form
$y=\sum_{j=0}^tr_j\beta^j$, where no other application of rules
(i) or (ii) is possible. We will show that the latter
representation satisfies the properties given in the lemma.

First let us show that if $m>p-1$, then the digits $m$ can appear
only in the suffix $m^j$ of $r_t\cdots r_0\bul$. Indeed, if there
is a pair of consecutive digits $m\,y$, where $y<m$, then it
occurs in a string $xm^ky$, where also $x<m$. But then, one can
use the rule (i) if $k=1$ or the rule (ii) if $k\geq 2$.

For $m\geq p-1$, denote by $j$ the minimal index such that
$r_j<m$. If $j>t$, then we set $c=0=u$ and the proof is finished.
If $j\leq t$ we set $c=r_j$ and $u=r_t\cdots r_{j+1}$. In order to
complete the proof, we need to verify that every suffix of $u$ is
lexicographically smaller than $(p-1)(p-2)^\omega$. This is indeed
true, since otherwise one can apply the rule (ii).
\end{proof}

\begin{pozn}\label{l:Katka1a}
Note that numbers $y$ with representation $y=u\bul$, where $u$
satisfies the properties 2.\ and 3.\ in Lemma~\ref{l:Katka1}, are
precisely all the non-negative $\beta$-integers in the R\'enyi
numeration system; their set is denoted by $\Z_\beta^+$,
see~\cite{BuFrGaKr1998}. It is shown there that
$$
X^{\lfloor\beta\rfloor}(\beta)\supset
\Z_\beta^+=\Sigma_\beta(\Omega)\cap [0,+\infty)\,,\quad
\text{where}\quad \Omega=[0,\beta)\,.
$$
This implies that for $y=u\bul$, one has $y'< \beta$.
\end{pozn}

The following corollary shows that the inclusion in Proposition~\ref{p:XY} cannot be replaced by equality
if only finitely many points are inserted to the spectrum.

\begin{coro}\label{c:1}
Let $\beta>1$ be a root of $x^2-px+1$, $p\geq 3$, and let
 $m\in\N$, $m\geq \lfloor\beta\rfloor$. For $k\geq 1$, define
 $y_k=\overline1m^k\bul$, where $\overline{1}$ stands for $-1$. Then
 $y_k\in\Sigma_\beta(\Omega)$, where $\Omega=\big[0,\frac{m\beta}{\beta-1}\big)$, but $y_k\notin
 X^m(\beta)$.
\end{coro}

\begin{proof}
We have $y_k=\overline1m^k \bul\in \Sigma_\beta(\Omega)$ since
 $$
 0 < m-\frac{1}{\beta}=y'_1\leq y'_k < \sum_{j=0}^\infty m (\beta')^j=\sum_{j=0}^\infty \frac{m}{
 \beta^j}=\frac{m\beta}{\beta-1}\,,
 $$
where we have used the fact that the sequence $(y_n')_{n\geq 1}$ is
increasing.

We now show that the assumption $y_k=\overline1m^k\bullet \in
X^m(\beta)$ leads to contradiction. For $k=1$, we have $y_1=m-\beta$ and
\begin{equation}\label{eq:nerov1}
y_1'=-\beta' +m = -\frac{1}{\beta} + m > m-1\,.
\end{equation}
If $y_1\in X^m(\beta)$, then by Lemma~\ref{l:Katka1} one can
represent $y_1$ by $y_1=ucm^j\bul$, with the required properties.
Since the digits of the new representation of $ucm^j$ are
non-negative, its last digit must be less or equal to $\lfloor y_1\rfloor=m-p$.   Thus necessarily $j=0$ and $c\leq m-p$, i.e.
$y_1=uc\bul$. Since by Remark~\ref{l:Katka1a}, we have $(u\bul)'<\beta$,
we can write
$$
y_1'=\big(\beta (u\bul) + c\big)' = \frac{(u\bul)'}{\beta} + c <
\frac{\beta}{\beta}+ m-p\leq m-2\,,
$$
This contradicts~\eqref{eq:nerov1}.

Suppose that for some $k\geq 2$, we have $y_k \in X^m(\beta)$. Let
$k$ be minimal with this property. We claim that the last digit of
the representation of $y_k$ from Lemma~\ref{l:Katka1} is equal to
$m$. Otherwise, the representation is of the form $y_k=uc\bul$,
where $(u\bul)'<\beta$ and $c\leq m-1$, i.e.
\begin{equation}\label{eq:nerov2}
y_k'=(\beta(u\bul)  + c)' = \frac{(u\bul)'}{\beta} + c \leq
\frac{\beta}{\beta}+m-1\leq m\,.
\end{equation}
On the other hand, we have $y_k=\overline{1}m^k$
$$
y'_k\geq y'_2 = (\overline{1}mm\bul)' =
-(\beta^2)'+m\beta'+m=\frac{-1+m\beta}{\beta^2}+m> m\,,
$$
which contradicts~\eqref{eq:nerov2}. Thus indeed, the
representation of $y_k$ from Lemma~\ref{l:Katka1} ends in $m$,
i.e.\ $y_k=ucm^j\bul$ with $j\geq 1$. Then we can consider
$y_{k-1}=\frac{y_k-m}{\beta}$ with a representation in the form
$ucm^{j-1}\bul$, which proves that $y_{k-1}\in X^m(\beta)$. This
contradicts the choice of $k$ as the minimal index such that
$y_{k}\in X^m(\beta)$.
\end{proof}

Now, we will consider the second class of quadratic unit bases,
namely such that $\beta^2=p\beta+1$, $p\geq 1$.

\begin{lem}\label{l:Katka2}
Let $\beta>1$ be a root of $x^2-px-1$, $p\geq 1$, and let
 $m\in\N$, $m\geq  \lfloor\beta\rfloor=p$. Then every $y\in
 X^m(\beta)$ has in base $\beta$ a
 representation $y=uwv\bul$, where
\begin{enumerate}
 \item $v$ is a (possibly empty) prefix of the infinite purely
 periodic word $(m0)^\omega$\,,
 \item $w=\epsilon$ or $w=cd$ where $c,d\in\{0,1,\dots,m-1\}$, and if $d\geq 1$, then $c\leq p-1$.
 \item $u\in \{0,1,\dots,p\}^*$\,.
\end{enumerate}
 \end{lem}

\begin{proof}
If $m=\lfloor\beta\rfloor=p$, then obviously, every element of the spectrum is in the form $y=u\bul$. Assume therefore that $m>\lfloor\beta\rfloor$.
The demonstration will use methods analogous to those of the proof
of Lemma~\ref{l:Katka1}. From $\beta^2=p\beta+1$, we can derive
the rewriting rules
\begin{itemize}
 \item[(i)] $z=0p1\bul\ \Rightarrow\ z=100\bul$\,,
 \item[(ii)] $z=0(p+1)00\bul \ \Rightarrow\
       z=10(p-1)1\bul$\,.
\end{itemize}
Note that the representation of $z$ on the right in the
implication in (i) has strictly smaller sum of digits than the
representation of $z$ on the left. In the rule (ii), both
representations have the same sum of digits, but the
representation on the right is strictly lexicographically greater
than that on the left.

Consider $y\in X^m(\beta)$. Repeated application of rules (i) and
(ii) yields a representation of $y$ in the form
$y=\sum_{j=0}^tr_j\beta^j$, where no other application of rules
(i) or (ii) is possible. We will show that the latter
representation satisfies the properties given in the lemma.

The final representation does not contain a substring of digits
$c\,d$ with $c\geq p$, $d\geq 1$. Otherwise, consider the most
left occurrence of such $c\,d$ in the string $r_tr_{t-1}\cdots
r_0$, i.e. one has a factor $xcd$, where $x<m$. Then one can use
the rule (i).

The representation $r_tr_{t-1}\cdots r_0\bul$ does not contain a
substring $m0d$ with $d<m$. By what has just been said, such a
substring occurs as a suffix of $xm0d$, where $x<m$. Then one can
use the rewriting rule (ii). (Note that at this point, we use the inequality
$m\geq p+1=\lfloor\beta\rfloor$+1.)

The only occurrence of the digit $m$ in the representation
$r_tr_{t-1}\cdots r_0\bul$ is in a suffix $v$, which is of the form
$(m0)^j$ or $(m0)^jm$, $j\geq 0$. Let $i$ be such that $r_t\cdots
r_0=r_t\cdots r_iv$, and the word $r_t\cdots r_i$ has only digits
in $\{0,1,\dots, m-1\}$. Moreover, if for some $k\geq i$, we have
$r_k>p$, then $k=i$ or $k=i+1$. Otherwise, we can use the
rewriting rule (ii).
\end{proof}

\begin{pozn}\label{pozn}
Note that numbers $y$ with representation $y=u\bul$ from
Lemma~\ref{l:Katka2} are in fact elements of the spectrum
$X^{\lfloor\beta\rfloor}(\beta)$. Since $\beta$ satisfying
$\beta^2=p\beta+1$, $p\geq 1$, is among the so-called confluent
Pisot numbers, we can use the result of~\cite{Frougny1992} that
any such $y$ is in fact a $\beta$-integer (forming the set
$\Z_\beta^+$). From~\cite{BuFrGaKr1998} we than have
$$
X^{\lfloor\beta\rfloor}(\beta)=\Z_\beta^+=\Sigma_\beta(\Omega)\cap[0,+\infty),\quad\text{where}
\quad\Omega=(-1,\beta)\,.
$$
This implies that if $y=u\bul$, then $-1<y'<\beta$.
\end{pozn}

Note that for the roots of $x^2-px-1$ we have equality between the $\beta$-integers and the spectrum $X^m(\beta)$ with  $m=\lfloor\beta\rfloor$, unlike the other class of quadratic Pisot units, see Corollary~\ref{c:1}. The following corollary therefore does not allow $m=\lfloor\beta\rfloor$.

\begin{coro}\label{c:2}
Let $\beta>1$ be a root of $x^2-px-1$, $p\geq 1$, and let
 $m\in\N$, $m> \lfloor\beta\rfloor$. For $k\in\N$ define
 $y_k=\overline1(0m)^k\bul$. Then
 $y_k\in\Sigma_\beta(\Omega)$ where $\Omega=\Big(-\frac{m\beta}{\beta^2-1},\frac{m\beta^2}{\beta^2-1}\Big)$, but $y_k\notin
 X^m(\beta)$.
\end{coro}

\begin{proof}
All numbers $y_k=\overline1(0m)^k\bul$, $k\in\N$, belong to $\Sigma_\beta(\Omega)$
since
$$
\begin{aligned}
 -1=y_0'<0 &< y_1'=-\frac1{\beta^2}+m  \leq y_k'=-(\beta')^{2k} + \sum_{i=0}^{k-1}
 m(\beta')^{2i}= \\
 &=-\frac{1}{(-\beta)^{2k}}+ \sum_{i=0}^{k-1}
\frac{m}{(-\beta)^{2i}}< \sum_{i=0}^\infty
\frac{m}{\beta^{2i}}=\frac{m\beta^2}{\beta^2-1}\,.
\end{aligned}
$$
Let us now show that $y_k\notin X^m(\beta)$ for all $k\in\N$.
First realize that $y_0=-1\notin X^m(\beta)\subset[0,+\infty)$.
For contradiction, suppose $k\geq 1$ is the minimal index such that $y_k\in X^m(\beta)$.
Find a representation of $y_k$ in the form $uwv\bul$ as given in Lemma~\ref{l:Katka2}. We will use that $-1<(u\bul)'<\beta$, as follows from Remark~\ref{pozn}.
Let us show that this representation does not have a suffix $v=(m0)^i$. Note that
$y_k'>0$ for $k\geq 1$. We may observe that numbers of the form
$y=ucdm0\bul$ have negative Galois conjugate, namely
$$
\begin{aligned}
y'&=(ucdm0\bul)'=\big((u\bul)\beta^4+c\beta^3 +d\beta^2 + m\beta\big)' =\\
&=
\frac{(u\bul)'}{(-\beta)^4}+\frac{c}{(-\beta)^3}+\frac{d}{(-\beta)^2}+\frac{m}{-\beta}<
\frac{\beta}{\beta^4}+\frac{m-1}{\beta^2}-\frac{m}{\beta}<0\,.
\end{aligned}
$$
For numbers of the form $ucd(m0)^i\bul$, $i\geq 1$, we have
$$
(ucd(m0)^i\bul)'\leq (ucdm0\bul)'<0\,.
$$
This proves that the representation of $y_k$ obtained from
Lemma~\ref{l:Katka2} is of the form $ucd(m0)^im\bul$ or $ucd\bul$.
The latter can be excluded by the following argument. Using
$(u\bul)'<\beta$, we get
\begin{equation}\label{eq:ctyri}
y_k'=\big(ucd\bul\big)'\leq \big(\beta^2(u\bul) + m-1\big)' =\frac{(u\bul)'}{\beta^2}
+ m-1 < \frac1\beta+m-1\,.
\end{equation}
On the other hand, since $k\geq 1$, we have
\begin{equation}\label{eq:pet}
y'_k\geq y'_1=(\overline{1}0m\bul)'
=m-\frac{1}{\beta^2}\,,
\end{equation}
which contradicts~\eqref{eq:ctyri}, since $m-1+\frac1\beta\leq m-\frac1{\beta^2}$ for every $\beta\geq \frac12(1+\sqrt5)$.
We have shown that the
representation of $y_k$ is of the form $ucd(m0)^im$ for some
$i\geq 0$. We next show that the before-last digit is indeed 0.
Otherwise, by item 2 of Lemma~\ref{l:Katka2}, the representation of $y_k$ is $ucdm\bul$ with $d\geq
1$, $c\leq p-1$. We estimate
\begin{equation}\label{eq:sest}
y_k'=(ucdm\bul)' = m -\frac{d}\beta + \frac{c}{\beta^2} -
\frac{(u\bul)'}{\beta^3} < m - \frac1\beta + \frac{p-1}{\beta^2} +
\frac1{\beta^3} = m-\frac1{\beta^2}\,,
\end{equation}
where we use that $(u\bul)'>-1$. Again, \eqref{eq:pet} contradicts
\eqref{eq:sest}. This proves that the representation of $y_k$ has suffix $0m$. It follows that the
number $y_{k-1}=\frac1{\beta^2}(y_k-m)$ has also a representation
in the base $\beta$ with non-negative digits, i.e. $y_{k-1}\in
X^m(\beta)$. This is a contradiction with the choice of $k$ as the
minimal index such that $y_{k}\in X^m(\beta)$.
\end{proof}

\section{Positional representations of numbers}\label{sec:representace}

Consider a real basis $\gamma$, $|\gamma|>1$, and a finite set $\A\ni 0$ of consecutive integers. An expression of a number $w\in\R$ in the form
$$
w=\sum_{i=0}^\infty\frac{D_i}{\gamma^i}, \quad D_i\in\A\,,
$$
is called a $(\gamma,\A)$-representation of $w$. We usually write $w=D_0\bullet D_1D_2D_3\cdots$. Denote ${\mathcal I}_{\gamma,\A}$ the set of real numbers $w$ having a $(\gamma,\A)$-representation. If the alphabet of digits is sufficiently large, then ${\mathcal I}_{\gamma,\A}$ is an interval.
More precisely, let $\A=\{a,\dots,0,1,\dots,A\}$, $a,A\in\Z$. If $A-a>|\gamma|-1$, then
\begin{equation}\label{eq:IgammaA}
{\mathcal I}_{\gamma,\A} = \begin{cases}
\Big[\frac{a\gamma}{\gamma-1},\frac{A\gamma}{\gamma-1}\Big] &\text{ if } \gamma>1\,,\\[2mm]
\Big[(A+a\gamma)\frac{\gamma}{\gamma^2-1},(A\gamma+ a)\frac{\gamma}{\gamma^2-1}\Big] &\text{ if } \gamma<-1\,.
\end{cases}
\end{equation}
A variant of the above statement for positive bases can be found
in~\cite{Pedicini}, it can, however, be simply verified by
checking that
$$
{\mathcal I}_{\gamma,\A} = \A + \frac1{\gamma}{\mathcal {\mathcal
I}_{\gamma,\A}}\,.
$$
This equality can be written in a more suitable way which gives an algorithm for finding a $(\gamma,\A)$-representation of a given number, namely,
\begin{equation}\label{eq:jakcifra}
\forall\, w\in{\mathcal I}_{\gamma,\A}\ \exists\, D\in\A \text{ and } \exists\, w_{\text{new}}\in{\mathcal I}_{\gamma,\A}\ \text{ such that } w=D+\frac{w_{\text{new}}}{\gamma}\,.
\end{equation}
The digit $D$ may not be given uniquely and almost all numbers have more than one $(\gamma,\A)$-representations.

In the rest of the section we provide an explicit prescription for a function $D:{\mathcal I}_{\gamma,\A}\to \A$ which allows us to find $(\gamma,\A)$-representations with specific properties. Besides the alphabet $\A$, we will use the alphabet $\B=\{b,\dots,0,1,\dots,B\}\subset\Z$
such that
\begin{equation}\label{eq:X}
a\leq b \leq 0 \leq B\leq A\quad\text{ and }\quad B-b=\big\lfloor|\gamma|\big\rfloor\,.
\end{equation}
Such an alphabet $\B$ has the minimal size ensuring that ${\mathcal I}_{\gamma,\B}$ is an interval. Obviously
$$
\B\subset\A \quad\text{ and }\quad {\mathcal I}_{\gamma,\B}\subset {\mathcal I}_{\gamma,\A}\,.
$$
We will give the prescription of the function $D$ separately for a
positive base $\gamma>1$ and a negative base
$\gamma<-1$. Both prescriptions ensure that one can find an
interval $I$ of length $|\gamma|$ for which $D(w)\in\B$ when $w\in
I$, and $I$ is an `attractor' of the transformation
$T(w):=\gamma\big(w-D(w)\big)$. In particular, for every $w$ from
the interior of ${\mathcal I}_{\gamma,\A}$ the iterations $T^k(w)$
belong to $I$ for sufficiently large $k$, and thus the
corresponding $(\gamma,\A)$-representation of $w$ contains
eventually only digits from the alphabet $\B$. This property is
formulated as Lemma~\ref{l:oatraktoru}.

We will use the notation $\ell,r$ for the left and right end-point of the interval ${\mathcal I}_{\gamma,\A}$, i.e. ${\mathcal I}_{\gamma,\A}=[\ell,r]$
where
$$
\begin{array}{llll}
\text{for the basis $\gamma>1$}   &\text{ one has } & \ell=\frac{a\gamma}{\gamma-1}\,, & r=\frac{A\gamma}{\gamma-1}\,,\\[2mm]
\text{for the basis $\gamma<-1$} &\text{ one has } & \ell=(A+a\gamma)\frac{\gamma}{\gamma^2-1}\,, & r=(a+A\gamma)\frac{\gamma}{\gamma^2-1}\,.
\end{array}
$$
For a parameter $L$ satisfying $\ell < L+a+1 \leq L+A <r$, we define
\begin{eqnarray*}
{\mathcal I}_a &=& [\ell,L+a+1)\,,\\
{\mathcal I}_k &=& [L+k,L+k+1)\,, \text{ for } a<k<A\,,\\
{\mathcal I}_A &=& [L+A,r]\,.
\end{eqnarray*}
Clearly, ${\mathcal I}_{\gamma,\A} = \bigcup_{k\in\A}{\mathcal I}_k$. The digit assignment $D: {\mathcal I}_{\gamma,\A} \to \A$ is then defined by
$$
D(w)=k \quad\text{ if }\quad w\in {\mathcal I}_{k}.
$$
Denote by $T$ the transformation
\begin{equation}\label{eq:T}
T(w)=\gamma\big(w-D(w)\big)\,.
\end{equation}
The following lemma specifies the value of $L$, so that $T$ is a
transformation on the interval ${\mathcal I}_{\gamma,\A}$, and
summarizes other useful properties. We do not include the proof,
which is straightforward.

\begin{lem}\label{l:oatraktoru}
With the above notation, denote by $I$ the interval
$I=\gamma\cdot[L,L+1)$ where
$$
L=\frac{b}{\gamma-1} \ \text{ if }\ \gamma>1 \qquad\text{ and }\qquad L=\frac{b-\gamma}{\gamma-1} \ \text{ if }\ \gamma<-1.
$$
Then
\begin{enumerate}
 \item[(i)] $T:{\mathcal I}_{\gamma,\A}\to{\mathcal I}_{\gamma,\A}$\,;
 \item[(ii)] $T(\I_k)\subset I$ for $a<k<A$\,;
 \item[(iii)] $I\subset \bigcup_{k\in\B}{\mathcal I}_k$\,;
 \item[(iv)] $T(I)\subset I$\,;
 \item[(v)] for every $w$ belonging to the interior of ${\mathcal I}_{\gamma,\A}$, there exists $k\in\N$ such that $T^k(w)\in I$.
\end{enumerate}
\end{lem}

%
%

\medskip
We will be interested in $(\gamma,\A)$-representations which are finite, i.e.\ ending in the suffix $0^\omega$. For numbers with finite representations, Lemma~\ref{l:oatraktoru} implies the following statements.

\begin{coro}\label{c:zacatekdukazu}
Let $w\in\I_{\gamma,\A}$.
\begin{enumerate}
\item If $T(w)$ has a finite $(\gamma,\A)$-representation, then $w$ has a finite $(\gamma,\A)$-representation.
\item Let $\gamma$ satisfy $\gamma\Z[\gamma]=\Z[\gamma]$. Then $w\in\Z[\gamma]\ \Leftrightarrow\ T(w)\in\Z[\gamma]$.
\item Let $\gamma$ satisfy $\gamma\Z[\gamma]=\Z[\gamma]$. If every $w\in I\cap\Z[\gamma]$ has a finite $(\gamma,\A)$-representation,
then every $w\in \I_{\gamma,\A}^\circ\cap\Z[\gamma]$ has a finite $(\gamma,\A)$-representation.
\end{enumerate}
\end{coro}

\section{Modified spectra as cut-and-project sets}\label{sec:modified}

From now on, we focus on spectra of quadratic units $\alpha=\pm\beta$, where $\beta>1$. In the proofs we will work with representations of numbers in base $\gamma=\frac1{\pm\beta'}$. Since $\beta$ is a unit, $\gamma$ is also a unit and
$$
\Z[\beta]=\Z[\gamma ]\quad\text{and}\quad \gamma\Z[\gamma]=\Z[\gamma]\,.
$$

The following proposition identifies the modified spectra $X^\A(\alpha)$ with cut-and-project sets. The acceptance interval $\Omega$ is in most cases given by the interior of
${\mathcal I}_{\frac1{\alpha'},\A}$. In fact, the non-zero boundary points of ${\mathcal I}_{\frac1{\alpha'},\A}$ have only infinite representation in base $\gamma=\frac1{\alpha'}$.
Set
\begin{equation}\label{eq:omega}
\Omega = {\mathcal I}_{\frac1{\alpha'},\A}^\circ\cup\{0\}\,.
\end{equation}

\begin{prop}\label{p:modified}
Let $\beta>1$ be a quadratic unit with conjugate $\beta'$. Let $\A\ni0$ be an alphabet of consecutive integers satisfying $\#\A>\beta$.
Let $\Omega$ be given by~\eqref{eq:omega}. If
$$
\alpha = -\beta\quad\text{ or }\quad
\alpha = \beta \text{ and } \{-1, 0,1\} \in\A\,,
$$
then $X^\A(\alpha) = \Sigma_\beta(\Omega)$.
\end{prop}

\pfz
Consider $(\gamma,\A)$-representations of numbers in $\Omega={\mathcal I}_{\gamma,\A}$, where $\gamma=\frac1{\pm\beta'}$.
Since the Galois automorphism $x\mapsto x'$ of the quadratic field $\Q(\beta)$ is a bijection on $\Z[\beta]=\Z[\gamma]$, equality
$X^\A(\pm\beta) = \Sigma_\beta(\Omega)$ can be equivalently rewritten as equality of the sets
$$
\begin{aligned}
\big(X^\A(\pm\beta)\big)' &= \big\{ \sum_{k=0}^na_k(\pm\beta')^k : n\in\N,\ a_k\in\A\big\} = \big\{\sum_{k=0}^n\frac{a_k}{\gamma^k} : n\in\N,\ a_k\in\A \big\}\,,\\
\big(\Sigma_\beta(\Omega)\big)' &= \{z'\in\Z[\gamma] : z'\in\Omega\} = \Z[\gamma] \cap \I_{\gamma,\A}\,.
\end{aligned}
$$
As $\frac1\gamma\in\Z[\gamma]$, we have $\sum_{k=0}^n\frac{a_k}{\gamma^k}\in\Z[\gamma]$, and therefore the inclusion $\big(X^\A(\pm\beta)\big)'$ $\subset\big(\Sigma_\beta(\Omega)\big)'$ is obvious.
In order to prove the opposite inclusion, we have to show that every $w\in\Z[\gamma] \cap \I_{\gamma,\A}^\circ$ has a finite
$(\gamma,\A)$-representation. By Corollary~\ref{c:zacatekdukazu}, it suffices to consider $w$ from the attractor $I$ of the transformation $T$, as defined in~\eqref{eq:T}, and show that there exists $n\in\N$ such that $T^n(w)$ has a finite
$(\gamma,\A)$-representation. We will even find $n$ such that $T^n(w)\in\B$, where $\B$ is the alphabet given in~\eqref{eq:X}.

We will construct a sequence $(w_k)_{k\in\N}$ by the prescription
$w_0=w\in\Z[\gamma]\cap I$ and $w_k=T(w_{k-1})$. Obviously
$w_k\in\Z[\gamma]\cap I$ for all $k\in\N$. We also have a sequence
$(z_k)_{k\in\N}$, where $z_k:=w_k'$. Since $w_{k-1}\in I$, by
Lemma~\ref{l:oatraktoru}, we have
\begin{equation}\label{eq:wn}
w_k=\gamma\big(w_{k-1}-D(w_{k-1})\big)\,,\quad\text{where }
D(w_{k-1})\in\B=\{b,\dots,B\}\,.
\end{equation}
In what follows, we will distinguish the proof of statement 1 and statement 2 of the proposition.

\smallskip
\noindent
{\bf Proof of statement 1.} We aim to show $X^\A(\beta) \supset \Sigma_\beta(\Omega)$. In this case, we have $\gamma=\frac1{\beta'}$, and therefore
\begin{equation}\label{eq:defzn}
z_k = \gamma'\big(w_{k-1}'-D(w_{k-1}')\big) =
\frac{z_{k-1}-D}{\beta}\,,\quad\text{where }
D\in\B=\{b,\dots,B\}\,.
\end{equation}
Let us study the relation of numbers $y$ and $y_{\text{new}}$, where $y_{\text{new}}$ is given by
$y_{\text{new}} = \frac1\beta(y-D)$ with $D\in\{b,\dots,0,\dots,B\}$.
One can easily check that
\begin{itemize}
\item
if $y\in\big[-\frac{B}{\beta-1},-\frac{b}{\beta-1}\big]$, then $y_{\text{new}}\in\big[-\frac{B}{\beta-1},-\frac{b}{\beta-1}\big]$;
\item
if $y>-\frac{b}{\beta-1}$, then $-\frac{B}{\beta-1}<y_{\text{new}}< y$;
\item
if $y<-\frac{B}{\beta-1}$, then $-\frac{b}{\beta-1}>y_{\text{new}}> y$.
\end{itemize}
This, together with~\eqref{eq:defzn} implies that there exists $n\in\N$ such that
\begin{equation}\label{eq:zn}
z_n\in \big[-\frac{B}{\beta-1},-\frac{b}{\beta-1}\big]\,.
\end{equation}
Since both $w_n$ and $z_n$ belong to $\Z[\beta]=\Z[\gamma]$, there exist $c,d\in\Z$ such that $w_n=c+d\beta$ and $z_n=c+d\beta'$.
We have
\begin{equation}\label{eq:dvenerovnosti}
w_n=c+d\beta\in I\quad\text{ and }\quad z_n=c+d\beta'\in\big[-\frac{B}{\beta-1},-\frac{b}{\beta-1}\big]\,.
\end{equation}
Now we need to distinguish between two cases, according to the minimal polynomial of~$\beta$.
\begin{itemize}
\item
Let $\beta^2=p\beta+1$, $p\in\N$, $p\geq 1$.  In this case $\gamma=\frac1{\beta'}=-\beta<-1$ and by Lemma~\ref{l:oatraktoru},
we have $I\subset[H,H+\beta]$, where $H=\beta\frac{b-1}{\beta+1}$. From~\eqref{eq:dvenerovnosti}, we obtain for $c,d\in\Z$ the inequalities
\begin{eqnarray}
H&\leq c+d\beta &\leq H+\beta\,, \label{eq:3hv1}\\
-\frac{B}{\beta-1} &\leq c+d\beta' &\leq -\frac{b}{\beta-1} \label{eq:3hv2}\,,
\end{eqnarray}
In order to find $c,d\in\Z$ satisfying~\eqref{eq:3hv1}
and~\eqref{eq:3hv2}, realize that from the definition of $H$ and
properties $b\leq -1$, $B\geq 1$, $B-b=\lfloor\beta\rfloor$, one
derives
\begin{equation}\label{eq:2hv}
H<0,\quad H+\beta>0,\quad -1<-\frac{B}{\beta-1}<0,\quad 0<-\frac{b}{\beta-1}<1\,.
\end{equation}
If $d\geq 1$, then~\eqref{eq:3hv1} implies that $c\leq -1$.
Therefore $c-\frac1\beta d\leq-1-\frac1\beta<-1$, which
contradicts~\eqref{eq:3hv2}. Similarly, if $d\leq -1$,
then~\eqref{eq:3hv1} implies that $c\geq 1$. Therefore
$c-\frac1\beta d\geq1+\frac1\beta>1$, which again
contradicts~\eqref{eq:3hv2}. We can therefore conclude that the
only pair of integers $c,d$ satisfying both~\eqref{eq:3hv1}
and~\eqref{eq:3hv2} is $c=d=0$, i.e.\ we necessarily have $w_n=0$, as desired.

\item
Let $\beta^2=p\beta-1$, $p\in\N$, $p\geq 3$.  In this case $\gamma=\frac1{\beta'}=\beta>1$ and by Lemma~\ref{l:oatraktoru},
we have $I\subset[H,H+\beta]$, where $H=\frac{b\beta}{\beta-1}$. From~\eqref{eq:dvenerovnosti}, we obtain for $c,d\in\Z$ the
same inequalities as~\eqref{eq:3hv1} and~\eqref{eq:3hv2}, which are again valid only for $c=d=0$.
\end{itemize}

\smallskip
\noindent
{\bf Proof of statement 2.}
We need to show that $X^{\A}(-\beta)\supset \Sigma_\beta(\Omega)$. We consider $\gamma=\frac1{-\beta'}$.
Thus from~\eqref{eq:wn}, we have
\begin{equation}\label{eq:defzn2}
z_k = \gamma'\big(w_{k-1}'-D(w_{k-1}')\big) =
\frac{z_{k-1}-D}{-\beta}\,,\quad\text{where }
D\in\B=\{b,\dots,B\}\,.
\end{equation}
Consider the relation of numbers $y$ and $y_{\text{new}}$, which satisfy $y_{\text{new}} = -\frac{y-D}{\beta}$, with $D\in\{b,\dots,0,\dots,B\}$. Denote $M=\max\{B,-b\}$. It can be easily checked that
\begin{itemize}
\item
if $|y|\leq\frac{M}{\beta-1}$, then $|y_{\text{new}}|\leq \frac{M}{\beta-1}$;
\item
if $|y|>\frac{M}{\beta-1}$, then $|y_{\text{new}}|< |y|$;
\item
if $y<\frac{b}{\beta+1}$, then $y_{\text{new}}> \frac{b}{\beta+1}$;
\item
if $y>\frac{B}{\beta+1}$, then $y_{\text{new}}< \frac{B}{\beta+1}$;
\end{itemize}
From the first two items, we can derive that eventually, $z_k\in
\big[-\frac{M}{\beta-1},\frac{M}{\beta-1}\big]$. The latter items
ensure that for some $n\in\N$, we have
\begin{equation}
\begin{array}{ll}
z_n\in \big[\frac{b}{\beta+1},\frac{B}{\beta-1}\big] &\quad\text{ if }M=B\,,\\
z_n\in \big[\frac{b}{\beta-1},\frac{B}{\beta+1}\big] &\quad\text{ if }M=-b\,.
\end{array}
\end{equation}
Fix such an $n$, and without loss of generality consider $M=B$.
If $c,d\in\Z$ are such that $w_n=c+d\beta$, then $c,d$ satisfy
\begin{eqnarray}
H&\leq c+d\beta &\leq H+\beta\,, \label{eq:4hv1}\\
\frac{b}{\beta+1} &\leq c+d\beta' &\leq \frac{B}{\beta-1} \label{eq:4hv2}\,,
\end{eqnarray}
where $H$ is chosen so that the attractor is equal to $I\subset[H,H+\beta]$. Its value again depends on the minimal polynomial of $\beta$.
It is not difficult to see that the only solution of such a system of inequalities are pairs $(c,d)=(0,0)$ and $(c,d)=(1,0)$, i.e.
$w_n=0$ or $w_n=1$.
\pfk

\begin{proof}[Proof of Theorem~\ref{t:2}]
It suffices to combine Propositions~\ref{p:modified} stating that
the modified spectrum is equal to a cut-and-project set, with
Proposition~\ref{p:cap3} which describes the distances between
consecutive elements of a cut-and-project set.
\end{proof}

\section{Distances in classical spectra}\label{sec:classical}

Let us return to the original question of classical spectra
$X^m(\beta)$ with positive base $\beta>1$ and alphabet of digits
$\A=\{0,1,\dots,m\}$. By Proposition~\ref{p:XY}, for quadratic
Pisot unit $\beta$, the spectrum is included in a cut-and-project
set, $X^m(\beta)\subset\Sigma_\beta(\Omega)$. Similarly as in the case of modified spectra, the acceptance interval
$\Omega$ is equal to ${\mathcal I}_{\gamma,\A}^\circ\cup\{0\}$, where $\gamma=\frac1{\beta'}$.

Unlike the case of modified spectra, by Proposition~\ref{p:nekonecneporuseni}, the inclusion $X^m(\beta)\supset\Sigma_\beta(\Omega)$ is
not valid for general $m$, even if we restrict ourselves to any
interval $[K,+\infty)$. Here, we aim to show that nevertheless,
one can conclude about the values of distances in the spectrum.
First we describe the elements of $\Sigma_\beta(\Omega)$ not
belonging to $X^m(\beta)$.


\begin{lem}\label{l:oS1}
Let $\beta>1$ be a root of $x^2-px+1$, $p\geq 3$, let $m\in\N$,
$m\geq \lfloor\beta\rfloor$. Let $\Omega$ be as in
Proposition~\ref{p:XY}. Then there exists a finite set ${\mathcal
S}$ such that for every $z>0$ satisfying
$$
z\in\Sigma_\beta(\Omega) \quad\text{and}\quad z\notin X^m(\beta)
$$
there exist $j\in\N$ and $s\in\S$ such that
$$
z=m\sum_{i=0}^{j-1}\beta^{i}+ s\beta^{j}\,.
$$
\end{lem}

\begin{proof}
Consider the notation of and before Lemma~\ref{l:oatraktoru} where
we have $\gamma=\frac1{\beta'}=\beta$ and $\A=\{0,1,\dots,m\}$, i.e.
$A=m$ and $a=b=0$. This implies that $L=0$ and the attractor of
$T$ is thus of the form $I=[0,\beta)$. By Remark~\ref{l:Katka1a},
if a positive $z\in\Z[\beta]$ satisfies $z'\in[0,\beta)=I$, then
$z\in X^{\lfloor\beta\rfloor}(\beta)\subseteq X^m(\beta)$.

Suppose that $z>0$, $z\in\Z[\beta]$, $z'\in\Omega=\I_{\gamma,\A}^\circ\cup\{0\}$
such that $z\notin X^m(\beta)$. In this case $z'$ does not have a
finite $(\gamma,\A)$-representation. Using the transformation
$T(w)=\gamma\big(w-D(w)\big)$ we again construct sequences
$(w_k)_{k\geq 0}$, $(z_k)_{k\geq 0}$ by the recurrence $w_0=z'$,
$w_k=T\big(w_{k-1}\big)$, and
$$
z_0=z>0,\quad z_k=w'_k = \frac{z_{k-1}-D(w_{k-1})}{\beta}\,.
$$
By the recurrence for $z_k$, we see that $z_k<z_{k-1}$ when
$z_{k-1}$ is positive. Since $w_k=z_k'\in \Omega$ for all
$k\in\N$, we have $z_k\in\Sigma_\beta(\Omega)$ which is a discrete
set. Therefore there exists an index $n\in\N$ such that
$z_n<0<z_{n-1}$. As $z_n=\frac1\beta(z_{n-1}-D(z_{n-1}'))<0$, we
derive $0<z_{n-1}<D(z_{n-1})\leq A$. In other words, $z_{n-1}$
belongs to the set
\begin{equation}\label{eq:S}
\S:=\big\{z\in\Z[\beta] : z\in (0,A), z'\in
\Omega\big\}\,.
\end{equation}
We have $\S=(0,A)\cap\Sigma_\beta(\Omega)$, which implies that
$\S$ is a finite set. By item 1 of
Corollary~\ref{c:zacatekdukazu}, we obtain $z_{k-1}\notin
X^m(\beta)\Rightarrow z_k\notin X^m(\beta)$, and thus for all
$k=0,1,\dots,n-1$ we have $z_k>0$ and $z_k\notin X^m(\beta)$. It
follows that $z_k'\notin I$ for any $k=0,1,\dots,n-1$. Item
(ii) of Lemma~\ref{l:oatraktoru} together with the fact that
$\I_0\subset I$, guarantee that $w_k=z_k'\in \I_A$ for every
$k=0,1,\dots,n-2$, whence $D(w_k)=m$ for every $k=0,1,\dots,n-2$.
Therefore
$$
w_0=\sum_{k=0}^{n-2}\frac{D(w_k)}{\beta^k} +
\frac{w_{n-1}}{\beta^{n-1}}\,.
$$
Realizing that $w_{n-1}'=z_{n-1}\in\S$, we obtain $z=w_0'$ in the
desired form.
\end{proof}

\begin{lem}\label{l:oS2}
Let $\beta>1$ be a root of $x^2-px-1$, $p\geq 1$, let $m\in\N$,
$m> \lfloor\beta\rfloor$. Let $\Omega$ be as in
Proposition~\ref{p:XY}. Then there exists a finite set ${\mathcal
S}$ such that for every $z>0$ satisfying
$$
z\in\Sigma_\beta(\Omega) \quad\text{and}\quad z\notin X^m(\beta)
$$
there exist $j\in\N$ and $s\in\S$ such that
$$
z=m\sum_{i=0}^{j-1}\beta^{2i}+ s\beta^{2j} \quad\text{or}\quad
z=m\sum_{i=0}^{j-1}\beta^{2i+1}+ s\beta^{2j+1}\,.
$$
\end{lem}

\begin{proof}
Consider $\gamma=\frac1{\beta'}=-\beta$, $\Omega=\I_{\gamma,\A}$.
We use the notation of Lemma~\ref{l:oatraktoru} with
$\A=\{0,1,\dots,m\}$, i.e. $A=m$ and $a=b=0$. We now have
$L=-\frac{\beta}{\beta+1}$ and the attractor of $T$ is thus of the
form $I=\big(-\frac{\beta}{\beta+1},\frac{\beta^2}{\beta+1}\big]$.
In particular, $I\subset (-1,\beta)$. By Remark~\ref{pozn}, if a
positive $z\in\Z[\beta]$ satisfies $z'\in(-1,\beta)$, then $z\in
X^{\lfloor\beta\rfloor}(\beta)\subseteq X^m(\beta)$.

Again, we construct sequences $(z_k)_{k\geq 0}$ and $(w_k)_{k\geq
0}$ by the recurrence $w_0=z'$, $w_k=T\big(w_{k-1}\big)$, and
$$
z_0=z>0,\quad z_k=w'_k = \frac{z_{k-1}-D(w_{k-1})}{\beta}\,.
$$
and find an index $n\in\N$ such that $z_n<0<z_{n-1}$. We define a
finite set $\S$ by the same prescription as in~\eqref{eq:S}. From
$z_0\notin X^m(\beta)$, we have $z_k\notin X^m(\beta)$ and thus
$z_k'\notin(-1,\beta)$. In particular, $z_k'\notin I$ for any
$k=0,\dots,n-1$.

By item (ii) of Lemma~\ref{l:oatraktoru}, we have $z_k'\in
\I_0\cup \I_A\setminus I$ for $k=0,1,\dots,n-2$. We deduce that
$D(z_k')=D(w_k)\in\{0,A\}$ for $k=0,1,\dots,n-2$. For concluding
the proof, it suffices to show that the digits $0$ and $A$
alternate, in particular, that for $k=0,1,\dots,n-2$ we have
$$
\begin{array}{lll}
D(w_k)=A &\Rightarrow & D(w_{k+1})\neq A\,, \\
D(w_k)=0 &\Rightarrow & D(w_{k+1})\neq 0\,.
\end{array}
$$
The first implication follows from $T(\I_A)\cap \I_A=\emptyset$.
For the second one, recall that $w_k\notin (-1,\beta)$ for
$k\in\{0,\dots,n-2\}$, and thus it suffices to verify
$T\big(\I_0\setminus (-1,\beta)\big)\cap \I_0=\emptyset$.

Altogether,
$$
w_0=\sum_{k=0}^{n-2}\frac{D(w_k)}{(-\beta)^k} +
\frac{w_{n-1}}{(-\beta)^{n-1}}
$$
and $z=w'_0$ has the required form.
\end{proof}

\begin{prop}\label{p:classical}
Let $\beta>1$ be a quadratic unit, $\A=\{0,1,\dots,m\}$,
$m\geq\lfloor\beta\rfloor$. Let $\Omega$ be as in
Proposition~\ref{p:XY}. Then for every $\delta>0$ there
exists $K>0$ such that
$$
[K,+\infty) \cap \Sigma_\beta\big((1-\delta)\Omega\big)
\subset X^m(\beta)\subset\Sigma_\beta(\Omega)\,.
$$
\end{prop}

\begin{proof}
The inclusion on the right is given by Proposition~\ref{p:XY}. In
order to prove the left inclusion, it suffices to show that for
every $\delta>0$ there exist only finitely many positive
$z\in\Z[\beta]$ such that $z\notin X^m(\beta)$ and
$z'\in(1-\delta)\Omega$. The constant $K$ is then chosen
bigger than maximum of such $z$.

Consider the case $\beta^2=p\beta+1$, $p\geq 1$. In this case
$\Omega=\big(\frac{-m\beta}{\beta^2-1},\frac{m\beta^2}{\beta^2-1}\big)=(\ell,r)$.
For positive $z\in\Z[\beta]$ not belonging to $X^m(\beta)$ with
$z'\in(1-\delta)\Omega\subset\Omega$ we can use
Lemma~\ref{l:oS2}, to derive that $z=x_j$ or $z=\beta x_j$, where
$$
x_j = m\frac{\beta^{2j}-1}{\beta^2-1}+s\beta^{2j}\,,
$$
for some $j\in\N$, $s\in\S$. For the Galois image of $x_j$, we
have
$$
x'_j =
\frac{m\beta^2}{\beta^2-1}+\frac{1}{\beta^{2j}}\Big(s'-\frac{m\beta^2}{\beta^2-1}\Big)\,.
$$
Denote $S=\max\Big\{|s'-\frac{m\beta^2}{\beta^2-1}|:s\in\S\Big\}$.
Thus
$$
x'_j>r-\frac{S}{\beta^{2j}}\quad\text{and}\quad (\beta x_j)' <\ell
+ \frac{S}{\beta^{2j+1}}\,.
$$
Indices $j\in\N$ satisfying
$$
r-\frac{S}{\beta^{2j}}<x_j'<(1-\delta)r \quad\text{ or }\quad
(1-\delta)\ell<(\beta x_j)' <\ell + \frac{S}{\beta^{2j+1}}
$$
are only finitely many, and hence also elements
$z\in\Sigma_\beta\big((1-\delta)\Omega\big)$ not belonging to
the spectrum $X^m(\beta)$ are finitely many.

The proof for the case $\beta^2=p\beta-1$, $p\geq 3$, is
analogous, using Proposition~\ref{l:oS1}.
\end{proof}

In the following section we will show that even if the spectrum is not equal to a cut-and-project set, Proposition~\ref{p:classical} allows us to state that the distances take only three values. We will provide these values in an explicit form together with their frequencies.

%

\section{Values of distances and frequencies}\label{sec:frequencies}

Sections~\ref{sec:modified} and~\ref{sec:classical} put into
connection the spectra (in general $X^\A(\alpha)$) with
cut-and-project sets. In order to determine the exact values of
distances in $X^\A(\alpha)$ and their frequencies, let us recall a
result of~\cite{MaPaPeGoslar}, providing a formula for the
distances between consecutive points of a cut-and-project sequence
$\Sigma_\beta(\Omega)=\{x\in\Z[\beta] : x'\in\Omega\}$, where
$\beta$ is a quadratic unit. We have $\beta\Z[\beta]=\Z[\beta]$,
and we can derive, directly from the definition, that
\begin{equation}\label{eq:scaling}
\begin{aligned}
\beta\Sigma_\beta(\Omega) &= \{\beta x \in \beta\Z[\beta]:
x'\in\Omega
 \}= \{y \in \Z[\beta]: \big(\tfrac{y}{\beta}\big)'\in\Omega \}=\\
 &= \{y \in \Z[\beta]: y'\in\beta'\Omega \} =
 \Sigma_\beta(\beta'\Omega)\,.
\end{aligned}
\end{equation}
Therefore it suffices to determine the distances in
cut-and-project sets with acceptance intervals of length
$|\Omega|$ for example within $(1,\beta]$. For cut-and-project
sets with other windows, the result can be simply derived by the
rescaling property~\eqref{eq:scaling}. Denote
$$
\phi_j(\beta):= \left\{\begin{array}{ll}
                \beta-j \,, & \hbox{for }j=0,1,\dots,\lfloor\beta\rfloor-1\,,\\
                1\,,        & \hbox{for }j=\lfloor\beta\rfloor\,.
                \end{array}\right.
$$
Note that $\phi_{j+1}(\beta)<\phi_j(\beta)$ and
$\bigcup_{j=1}^{\lfloor\beta\rfloor} \left( \phi_j(\beta),
\phi_{j-1}(\beta) \right] = (1,\beta]$. With this notation we can
cite the following result of~\cite{MaPaPeGoslar}.

\begin{prop}\label{p:mezerycap}
Let $\beta$ be a quadratic Pisot unit. The distances between
consecutive points of the cut-and-project set
$\Sigma_\beta(\Omega)$ take the following values:
\begin{itemize}
\item If $|\Omega|\in\big(\phi_j(\beta),\phi_{j-1}(\beta)\big)$,
$j=1,\dots,\lfloor\beta\rfloor$, then the distances are $1$,
$j-\beta'$ and $j+1-\beta'$;

\item if $|\Omega|=\phi_{j-1}(\beta)$,
$j=1,\dots,\lfloor\beta\rfloor$, and $\Omega$ is a semi-closed
interval, then the distances take two values, namely $1$,
$j-\beta'$.
\end{itemize}
\end{prop}

As a consequence, we can provide the proof of Theorem~\ref{t:1}.

\begin{proof}[Proof of Theorem~\ref{t:1}]
As a result of Proposition~\ref{p:classical}, for each
$\delta>0$ there exist $K>0$ such that
$$
[K,+\infty) \cap \Sigma_\beta\big((1-\delta)\Omega\big)
\subset X^m(\beta)\subset\Sigma_\beta(\Omega)\,.
$$
Combining Proposition~\ref{p:cap3} with
Proposition~\ref{p:mezerycap}, we can see that for sufficiently
small $\delta$, the distances between consecutive points in
$\Sigma_\beta\big((1-\delta)\Omega\big)$ and in
$\Sigma_\beta(\Omega)$ take the same three values. Necessarily,
the same three values are taken by distances between elements of
the spectra.
\end{proof}

Apart the values of distances, one can be interested in the
frequency of occurrence of the given value in the gap sequence.
Formally, if the gap sequence is coded by an infinite word
$u=u_0u_1u_2\cdots$ over a finite alphabet ${\mathcal B}$ formed
by symbols each corresponding to a different value of the
distance, then the frequency of the symbol $X\in\B$ in $u$ is
given as the limit
$$
\rho_X=\lim_{n\to\infty} \frac{\#\{0\leq i<n : u_i=X\}}{n}\,,
$$
if it exists. It is a well known fact that frequencies of letters
$A,B,C$ in 3iet words, as defined in Section~\ref{sec:cap}, are
given by the lengths of intervals $I_A=[0,\alpha)$,
$I_B=[\alpha,\beta)$, $I_C=[\beta,1)$ in the corresponding
transformation $T$, cf.\ Definition~\ref{d:3iet}, namely
$$
\rho_A=\alpha\,,\quad \rho_B=\beta-\alpha\,,\quad
\rho_C=1-\beta\,.
$$

The following theorem summarizes the results about exact values of
distances and their frequencies in both classical and modified
spectra. Note that the mentioned finitely many exceptions apply
only in the case of classical spectrum.

\begin{thm}\label{t:explicit}
Let $\beta$ be a quadratic Pisot unit, and let $\A\ni0$ be a
finite set of consecutive integers, $\#\A>\beta$. If
$$
\beta^{l-1}(\beta-1)\cdot\max\{\beta-j,1\}<\#\A-1
\leq\beta^{l-1}(\beta-1)(\beta-j+1)
$$
for $l\geq 0$ and $j\in\{1,\dots,\lfloor\beta\rfloor\}$, then, up
to finitely many exceptions, the distances in $X^\A(\pm\beta)$
take values
$$
\frac1{\beta^l},\quad \frac1{\beta^l}(j-\beta'),\quad
\frac1{\beta^l}(j+1-\beta')
$$
with frequencies
$$
1-\frac{\beta^{l-1}(\beta-1)}{\#\A-1},\quad
1-\frac{\beta^{l-1}(\beta-j)(\beta-1)}{\#\A-1},\quad -1+
\frac{\beta^{l-1}(\beta-j+1)(\beta-1)}{\#\A-1},
$$
where $\beta'$ denotes the Galois conjugate of $\beta$.
\end{thm}

\begin{proof}
Let us first show the statement for the modified spectrum
$X^\A(\alpha)$, $\alpha=\pm\beta$, satisfying assumptions of
Theorem~\ref{t:2}. By Proposition~\ref{p:modified}, the modified
spectrum is equal to the cut-and-project set
$\Sigma_\beta(\Omega)$ for some $\Omega$.
By Proposition~\ref{p:mezerycap}, the distances in
$\Sigma_\beta(\Omega)$ depend only on the length of the interval.
For both $\alpha=\beta$ and $\alpha=-\beta$ we have
$$
|\Omega|=|{\mathcal I}_{\frac1{\beta'}}| = |{\mathcal I}_{-\frac1{\beta'}}|
= \frac\beta{\beta-1}(A-a) = \frac{\beta(\#\A-1)}{\beta-1}\,,
$$
as can be verified for both $\beta^2=p\beta+1$ and
$\beta^2=p\beta-1$ using~\eqref{eq:IgammaA}. By
Proposition~\ref{p:mezerycap} and the scaling
property~\eqref{eq:scaling}, we need to find $l\in\N$ for which
$\beta^l<|\Omega|\leq \beta^{l+1}$, and find to which interval
$\big(\phi_j(\beta),\phi_{j-1}(\beta)\big]$ the length
$|\Omega|\beta^{-l}$ belongs.

The frequencies are derived easily by finding the parameters
$\alpha$, $\beta$ of intervals in the corresponding exchange of
three intervals. The identification of cut-and-project set with
3iet words is given in Proposition~\ref{p:cap3}. Recall that the
isomorphism $\star$ applied on $\Delta_1,\Delta_2\in\Z[\beta]$ is
now the Galois automorphism in the field $\Q(\beta)$.

In order to conclude the demonstration of the statement for the
classical spectra $X^m(\beta)$, it suffices to realize that by the
proof of Theorem~\ref{t:1}, the values of distances (up to
finitely many exceptions) in $X^m(\beta)$ coincide with those in
$\Sigma_\beta(\Omega)$, where again $|\Omega|=|{\mathcal
I}_{\frac1{\beta'}}|$.
\end{proof}

\section{Comments}\label{sec:remarks}

\begin{itemize}
\item In Section~\ref{sec:representace}, we have considered
representations of numbers in general bases $\gamma$, $|\gamma|>1$
and arbitrary alphabets of consecutive integers containing 0. Let
us mention that in case of positive base and the digit set
$\A=\{0,1,\dots,m\}$, $m>\gamma-1$, the transformation $T$ given by
Lemma~\ref{l:oatraktoru} restricted to the attractor is
 homothetic to the transformation $t:[0,1)\to[0,1)$ providing the
 greedy expansion according to R\'enyi~\cite{Renyi}.

 Similarly, if the base is $\gamma=-\beta<-1$, and the alphabet is
$\A=\{0,1,\dots,m\}$, $m>\beta-1$, then $T$ restricted to the
attractor corresponds to the transformation
 $t:\big[-\frac{\beta}{\beta+1},\frac1{\beta+1}\big)\to\big[-\frac{\beta}{\beta+1},\frac1{\beta+1}\big)$, as considered by Ito and Sadahiro~\cite{ItoSadahiro}.


\item Algorithms for arithmetic operations in systems with base
$\alpha$ and digit set $\A\subset\Z$, $0\in\A$, usually work with
numbers having finitely many non-zero digits, formally, belonging
to the set
$$
{\rm fin}_\A(\alpha)=\Big\{\sum_{i\in J}a_i\alpha^i :
J\subset\Z,\, J\text{ finite},\, a_i\in \A\Big\} =
\bigcup_{k\in\N}\alpha^{-k}X^\A(\alpha)\,.
$$
If $\alpha$ is an algebraic unit, then ${\rm fin}_\A(\alpha)
\subset \Z[\alpha]$. Essential is the knowledge whether ${\rm
fin}_\A(\alpha)$ is closed under addition and subtraction.

It can be derived that whenever $\alpha=\pm\beta$, $\beta>1$, is a
quadratic unit and $X^\A(\alpha)=\Sigma_\beta(\Omega)$, then ${\rm
fin}_\A(\alpha)$ is closed under addition. If, moreover, $0$ lies
in the interior of the interval $\Omega$, then ${\rm
fin}_\A(\alpha) = \Z[\alpha]$ and thus ${\rm fin}_\A(\alpha)$ is
closed under both, addition and subtraction. The cases of $\alpha$
and $\A$ where this happens can be read in
Proposition~\ref{p:modified}.
\end{itemize}

\section*{Acknowledgements}

This work was supported by the Czech Science Foundation, grant No.\ 13-03538S. The second author acknowledges financial support of
the Grant Agency of CTU in Prague, grant No.\ SGS11/162/OHK4/3T/14.


\end{document}